\documentclass[a4paper,10pt]{amsart}
\usepackage{amsmath,amssymb,amsthm,amscd}

\newcommand{\add}{\mathrm{add}}
\newcommand{\Add}{\mathrm{Add}}
\newcommand{\Prod}{\mathrm{Prod}}
\newcommand{\Gen}{\mathrm{Gen}}
\newcommand{\Cogen}{\mathrm{Cogen}}

\newcommand{\rmod}{\mathrm{Mod-}}
\newcommand{\rfmod}{\mathrm{mod-}}

\newcommand{\lm}{{\mbox{\rm $R$-mod}}}
\newcommand{\LM}{{\mbox{\rm $R$-Mod}}}
\newcommand{\M}{\mbox{\rm Mod-$R$}}
\newcommand{\m}{\mbox{\rm mod-$R$}}

\newcommand{\Z}{\mathbb{Z}}

\newcommand{\N}{\mathbb{N}}

\newcommand{\mapr}[1]{\xrightarrow{#1}}

\newcommand{\exs}[5]{0\rightarrow #1 \mapr{#2} #3 \mapr{#4} #5 \rightarrow 0}

\newcommand{\Ext}[3]{\mbox{Ext}^1_{#1}\,(#2,#3)}

\newcommand{\Exti}[4]{\mbox{Ext}^{#1}_{#2}\,(#3,#4)}

\newcommand{\Tori}[4]{\mbox{Tor}_{#1}^{#2}\,(#3,#4)}

\DeclareMathOperator{\HomOp}{Hom}

\newcommand{\Hom}[3]{\HomOp_{#1}(#2,#3)}

\newcommand{\Acal}{\ensuremath{\mathcal{A}}}
\newcommand{\Pcal}{\ensuremath{\mathcal{P}}}

\newcommand{\Rcal}{\ensuremath{\mathcal{R}}}
\newcommand{\Fcal}{\ensuremath{\mathcal{F}}}
\newcommand{\Gcal}{\ensuremath{\mathcal{G}}}
\newcommand{\Dcal}{\ensuremath{\mathcal{D}}}
\newcommand{\Scal}{\ensuremath{\mathcal{S}}}

\newcommand{\Mcal}{\ensuremath{\mathcal{M}}}
\newcommand{\Lcal}{\ensuremath{\mathcal{L}}}
\newcommand{\Ecal}{\ensuremath{\mathcal{E}}}
\newcommand{\Ccal}{\ensuremath{\mathcal{C}}}
\newcommand{\Qcal}{\ensuremath{\mathcal{Q}}}
\newcommand{\p}{\ensuremath{\mathbf{p}}}
\newcommand{\comp}{\ensuremath{\mathbf{c}}}
\newcommand{\q}{\ensuremath{\mathbf{q}}}
\newcommand{\tube}{\ensuremath{\mathbf{t}}}

\newcommand{\calb}{\ensuremath{\mathcal{B}}}
\newcommand{\ra}{\rightarrow}
\theoremstyle{plain}
\newtheorem{thm}{Theorem}
\newtheorem{prop}[thm]{Proposition}
\newtheorem{lem}[thm]{Lemma}
\newtheorem{lemma}[thm]{Lemma}
\newtheorem{cor}[thm]{Corollary}
\newtheorem{ex}[thm]{Example}
\theoremstyle{definition}

\theoremstyle{remark}
\newtheorem*{rem}{Remark}

\begin{document}
\title{Large tilting modules and representation type}
\author{\textsc{Lidia Angeleri H\" ugel}}
\address{Dipartimento di Informatica e
Comunicazione,\\ Universit\`a degli Studi dell'Insubria\\
Via Mazzini 5, I - 21100 Varese, Italy}
\email{lidia.angeleri@uninsubria.it}
\author{\textsc{Otto Kerner}}
\address{Mathematisches Institut, Heinrich-Heine-Universit\" at D\" usseldorf\\
Universit\" atsstr.1, 40225 D\" usseldorf, Germany}
\email{kerner@math.uni-duesseldorf.de}
\author{\textsc{Jan Trlifaj}}
\address{Charles University, Faculty of Mathematics and Physics, Department of Algebra \\
Sokolovsk\'{a} 83, 186 75 Prague 8, Czech Republic}
\email{trlifaj@karlin.mff.cuni.cz}
\thanks{We acknowledge
support by  Universit\`a di Padova, Progetto di Ateneo
CDPA048343. First author also partially supported  by PRIN 2005 "Prospettive in teoria degli anelli, 
algebre di Hopf e categorie di moduli",  by the DGI and the
European Regional Development Fund, jointly, through Project
 MTM2005--00934, and by the Comissionat per Universitats i Recerca
of the Generalitat de Ca\-ta\-lunya, Project 2005SGR00206. Third author
 acknowledges support by GA\v CR 201/06/0510 and MSM 0021620839.}
%\subjclass[2000]{...{primary), ... (secondary)}
\date{\today}

% = Abstract ========================================================
\begin{abstract}
We study finiteness conditions on large tilting modules over arbitrary rings. We then turn to a hereditary artin algebra $R$ and apply our results to the (infinite dimensional) tilting module $L$ that generates all modules without preprojective direct summands. We show that the behaviour of $L$ over its endomorphism ring determines the representation type of $R$. A similar result holds true for the  (infinite dimensional) tilting module $W$ that generates the divisible modules. Finally, we extend to the wild case some results on Baer modules and torsion-free modules proven in \cite{baerml} for tame hereditary algebras.
\end{abstract}

\maketitle
%\vspace{4ex}

%%%%%%%%%%%%%%%%%%%%%%%%%%%%%%%%%%%%%%%%%%%%%%%%%%%%%%%%%%%%%%%%%%%%%%%
\section*{Introduction}
The category mod-$R$ of all finitely generated modules over a hereditary artin algebra $R$ is well understood. Let us briefly recall its main properties. First, every finitely generated $R$-module has an essentially unique indecomposable decomposition. Further, the finitely generated indecomposable modules are depicted in the Auslander-Reiten quiver of $R$. If $R$ is indecomposable and has infinite representation type, 
this quiver has the following shape
%%%%%%%%%%%%%%%%%%%%%%%%%%%%

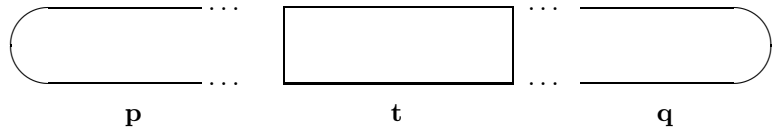
\begin{center}
%%%%%%%%%%%%%%%%%%%%%%%%%%%%%%%%%%%%%%%%%%%%%
\setlength{\unitlength}{1mm}
\begin{picture}(100,15)
%%%  preprojectives %%%%%%%%%%%%%%%%%%%%%%%%%%%
\put(20,5){\oval(50,10)[l]}
\put(21,-0.3){\ldots}
\put(21,9.7){\ldots}
%%%%  preinjectives%%%%%%%%%%%%%%%%%%%%%
\put(70,5){\oval(50,10)[r]}
\put(63,-0.3){\ldots}
\put(63,9.7){\ldots}
%%%%%%  regulars  %%%%%%%%%%%%%%%%%%%%%%%
\put(31,0){\framebox(30,10)}
%%%%%%%%%  names %%%%%%%
\put(10,-5){\p}
\put(45,-5){\tube}
\put(80,-5){\q}
\end{picture}
\vspace{.9 cm}
%%%%%%%%%%%%%%%%%%%%%%%%%%%%%%%%%%%%%%%%%%%%%%%%%%%%%%%%%%%%%%%%%%%%%%%%%%%%%%%
\end{center}
where
$\p$ contains all indecomposable projectives and is called the preprojective component, 
$\q$ contains all indecomposable injectives and is called the preinjective component, and
 $\tube$  consists of infinitely many infinite components, called regular components.

\medskip

Much less is known about the category Mod-$R$ of all $R$-modules, if $R$ is representation infinite. 
In his seminal paper \cite{R} from 1979, Ringel initiated the study of the infinite dimensional modules by investigating some torsion pairs in Mod-$R$ constructed from the Auslander-Reiten components of $R$.

For example, he considered the torsion pair $(\Rcal, \Dcal)$  cogenerated by $\tube$. It provides a cut of Mod-$R$ into a torsion-free class $\Rcal$ containing $\p$ and $\tube$, and a torsion class $\Dcal$ containing $\q$, and in some sense, it is maximal with respect to this property, see \cite{RR} and \ref{extremal}. When $R$ is of tame representation type, the torsion pair $(\Rcal, \Dcal)$ splits, and in view of the striking analogies with the category of abelian groups, the modules in $\Dcal$ are called divisible, while the modules in $\Rcal$ are called reduced. 

Ringel also considered the torsion pair $(\Pcal, \Lcal)$ generated by $\p$. Here the torsion-free class $\Pcal$ contains $\p$, the torsion class $\Lcal$ contains $\tube$ and $\q$, and again, the torsion pair is maximal with respect to this property in the sense of \ref{extremal}. However,  $(\Pcal, \Lcal)$ is not a split torsion pair unless $R$ is of finite representation type. 

Finally, there are also the dual constructions: the torsion pair $(\Fcal, \Gen\tube)$ generated by $\tube$, and the torsion pair $(\Ccal, \Qcal)$ generated by $\q$, see \cite{R}, or \ref{torsionfree} and  \ref{preinjective}.

\medskip

The aim of our paper is to study these torsion pairs from the point of view of infinite dimensional tilting theory. 
Indeed, there are tilting modules $W$ and $L$ such that $\Dcal=\Gen W$ and $\Lcal=\Gen L$. If $R$ is tame, then it is shown in \cite{RR} that $W$
can be chosen as the direct sum of a set of representatives of the Pr\"ufer modules and the generic module $G$.
A construction of $W$ in the wild case,  as well as a construction of $L$ in case $R$ has infinite representation type, can be found in the works of Lukas \cite{L1, L2}; for more details we refer to the paper \cite{KT}.

\medskip
It turns out that $W$ and $L$ play a remarkable role both in the tame and in the wild case.  Indeed, they control the behaviour of the category Mod-$R$: one can read off the representation type of $R$ from finiteness conditions satisfied by $W$ or by $L$. For example, $R$ is of tame representation type if and only if $L$ is noetherian when viewed as a module over its endomorphism ring End$L$. Moreover, if $L$ has finite length over End$L$, then $R$ has finite representation type (Theorems \ref{tame} and \ref{frt}).

\medskip
These results are applications of more general investigations carried out in the first part of the paper. We consider arbitrary tilting modules over an arbitrary ring $R$. As explained in Section 1, using results from  \cite{AHT, bazher, bast}, every tilting class $T^\perp$ in Mod-$R$ corresponds  
bijectively to a resolving subcategory $\Scal$ of mod-$R$, and also to a cotilting class ${}^\perp C$ in the category of left $R$-modules $R$-Mod. This allows us to associate to $T$ cotorsion pairs in Mod-$R$ and $R$-Mod. In  Section 2, we characterize finiteness conditions on $T$ in terms of these cotorsion pairs and of the resolving subcategory $\Scal$.

In Section 3, we restrict to the case where $T$ has projective dimension one. Then $T$ gives rise to a torsion pair in Mod-$R$ with torsion class $\Gen T$. If $R$ is a hereditary artin algebra, and $\Scal$ is a union of Auslander-Reiten-components, also the torsion pair $(\Fcal, \Gcal)$ in Mod-$R$ with torsion-free class $\Fcal=\varinjlim \Scal$ is of importance. We prove that $T$ is product-complete if and only if $(\Fcal, \Gcal)$ is a split torsion pair (Corollary \ref{splittp}).  

In fact, the latter result is a consequence of our investigations in Section 4 devoted to the class $\calb$ of all Baer modules for the torsion class $\Gcal$.
Recall that $\calb$ is the class of all modules $M$ such that  $\Ext RMG = 0$ for all $G \in \Gcal$. In Theorem \ref{filtbaer} we show that a module belongs to $\calb$ if and only if it is $\Scal$-filtered, generalizing a result from \cite{baerml}. 

Finally, in Section 5, we apply our results to the case where $R$ is a hereditary artin algebra, and $\Scal=\add \p$. This enables us to prove in Section 6 that the tilting modules $L$ and $W$ determine the representation type of $R$. Moreover, 
we give an alternative proof of Ringel's result  \cite[3.7 - 3.9]{R} stating that $R$ is tame if and only if the torsion pair $(\Ccal, \Qcal)$ splits.
We also extend to the wild case some results on Baer modules and torsion-free modules obtained in \cite{baerml}. 

\bigskip

%%%%%%%%%%%%%%%%%%%%%%%%%%%%%%%%%%%%%%%%%%%%%%%%%%%%%%%%%%%%%%%%%%%%%%%%%%%%%%%%%%%%%%%%%%%%%%%%%%%%%%%%%%%%%%%%%%%%%%%%%%%%%%%%%%%%%%%
\section{Preliminaries} 

\subsection{Notation.} Let $R$ be a ring, and let \M\ and \LM\ be the categories of all right and left $R$-modules, respectively. We denote by \m\ the subcategory of all modules possessing a projective resolution consisting of finitely generated modules, and we define {\lm} correspondingly.

\medskip
For a right $R$-module $M$, we denote by $M^\ast=\Hom{\Z}{M}{{\mathbb{Q}/\Z}}$ its character (left $R$-) module. 
Instead of the character module we can equivalently consider other dual modules, for example, 
for modules over an artin algebra, we can take $M^\ast=D(M)$ where $D$ denotes the standard duality. 
If $\Scal$ is a class of modules, we denote by $\Scal^\ast$ the corresponding class
of all duals $B^\ast$ of the modules $B\in\Scal$.

\medskip
For a class of modules $\mathcal C$, we denote $\Ccal^{<\omega}=\Ccal\cap\m.$ Moreover, we define  
$$^o \mathcal C = \{ M \in \rmod R \mid \Hom RMC = 0 \mbox{ for all } C \in
\mathcal C \},$$
$$^{\perp_1} \mathcal C = \{ M \in \rmod R \mid \Ext RMC = 0 \mbox{ for all } C \in
\mathcal C \}$$
and
$${^\perp} \mathcal C = \{ M \in \rmod R \mid \Exti iRMC = 0 \mbox{ for all } C \in
\mathcal C \mbox{ and all } i > 0 \}$$
$$\mathcal C ^\intercal = \{ M \in \LM \mid \Tori iRCM = 0 \mbox{ for all } C \in
\mathcal C \mbox{ and all } i > 0 \}.$$
Similarly, the classes $\mathcal C ^o$, $\mathcal C ^{\perp_1}$, $\mathcal C ^\perp$, and $^\intercal \mathcal C$ are
defined.

We denote by  {Add}\,{$\mathcal C $} (respectively,  {add}\,{$\mathcal C $}) the class consisting of all modules isomorphic to direct summands of (finite) direct sums of modules of ${\mathcal C } $. The class  consisting of   all modules isomorphic to direct  summands of direct products of modules of ${\mathcal C } $ is denoted by {Prod}\,{$\mathcal C $}.  Finally, Gen{$\mathcal C $} and Cogen{$\mathcal C $} denote the class of all modules generated and cogenerated, respectively, by the modules in ${\mathcal C } $.

We will say that a module $M_R$ with the endomorphism ring $S$ is {\em endonoetherian} if $M$ is noetherian when viewed as a left $S$-module.
If $_SM$ has finite length then $M$ is called {\em endofinite}.
Finally, following \cite{KSa}, a module $M$ with Add$M$ closed under direct products will be called
\emph{product-complete}. Note that $M$ is product-complete iff $\text{Add}M=\text{Prod}M$.
Moreover, every product-complete module is $\Sigma$-pure-injective.

\medskip
\subsection{Tilting and cotilting cotorsion pairs.}
A module $T$ is said to be  a ($n$-) {\em tilting module} if it 
 satisfies
 
  (T1) proj.dim$(T) \leq n$; 
  
  (T2) $\Exti iRT{T^{(I)}} = 0$ for each set $I$ and each $i > 0$; and 
  
  (T3) there are $r\in\N$ and an    exact sequence 
$0 \to R \to T_0 \to T_1 \to\ldots \to T_r\to 0$ where $T_i \in \mbox{Add}(T)$ for all $i \leq r$.

The class $T^\perp$ is then called the {\em tilting class} induced by $T$.

\medskip 

Note that $T$ is a $1$-tilting module if and only if the class  $T^\perp$  coincides with Gen$T$. One then has a {\it tilting torsion pair}
$(T^o, \text{Gen}T)$  \footnote{We adopt the convention of writing the torsion-free class on the left side of the torsion pair.}.
The inclusion $T^\perp \subseteq \text{Gen}T$ holds true for any $n$-tilting module $T$, see \cite[2.3]{AC1}.
\medskip

{\it Cotilting modules} and {\em classes} are defined dually and have the dual properties.

\medskip

 Tilting and cotilting classes arise naturally in cotorsion pairs. 
A {\em cotorsion pair} is a pair of classes of modules $(\mathcal A,\mathcal B)$ such that
$\mathcal A = {}^{\perp_1} \mathcal B$ and $\mathcal B = \mathcal A ^{\perp_1}$.
If $\Scal$ is a class of right $R$-modules, we obtain    a cotorsion pair $(\mathcal A,\mathcal B)$ by setting
$\mathcal B = \Scal ^{\perp_1}$
and  $\mathcal A = {}^{\perp_1 }(\Scal ^{\perp_1 })$. It is called the cotorsion pair 
{\em generated}\footnote{Our terminology follows \cite{GT}, hence it differs from  previous use.}  by $\Scal$.
Dually, if $\Scal$ is a class of right $R$-modules, we obtain    a cotorsion pair $(\mathcal A,\mathcal B)$ by setting
$\mathcal A= {}^{\perp_1 }\Scal $ and $\mathcal B = (^{\perp_1 }\Scal) ^{\perp_1}$. It is called the cotorsion pair {\em cogenerated} by $\Scal$.

A cotorsion pair $(\mathcal A,\mathcal B)$ is said to be \emph{complete} if
for every module $X$ there are short exact sequences $0\to X\to B\to
A\to 0$ and $0\to B'\to A'\to X \to 0$ where $A,A'\in\mathcal A$ and
$B,B'\in\mathcal B$.
 Cotorsion pairs generated by a set of modules or cogenerated by a class of pure-injective modules are always complete \cite[3.2.1 and 3.2.9]{GT}.

\medskip

Cotorsion pairs  $(\mathcal A,\mathcal B)$ with $\mathcal B=T^\perp$ for some   $n$-tilting module $T$ are called $n$-{\em tilting cotorsion pairs}, and cotorsion pairs  $(\mathcal A,\mathcal B)$ with $\mathcal A={}^\perp C$ for some   $n$-cotilting module $C$ are called $n$-{\em cotilting cotorsion pairs}. We are now going to describe them as cotorsion pairs generated, respectively cogenerated, by certain classes of modules.

\medskip

Recall that a subcategory $\Scal$ of \m\ is said to be {\it resolving}, if it is closed under direct summands,  extensions,  kernels of epimorphisms, and contains $R$. If $\Scal$ is resolving, then $\Scal^\perp = \Scal^{\perp_1}$, and ${}^{\perp_1}(\Scal^\perp )=
{}^{\perp}(\Scal^\perp )$, see \cite[2.2.11]{GT}.

Dually,  we denote by ${\mathcal P \mathcal I}$ the full subcategory of $\M$ consisting of the pure-injective modules, and we say that   a subcategory $\Scal$ of ${\mathcal P \mathcal I}$ is {\it coresolving}  if it is closed under direct summands,  extensions, cokernels of monomorphisms, and contains all the injective modules. Moreover, for $\Scal$ coresolving, ${}^{\perp_1}\Scal = {}^\perp\Scal$,
and $({}^\perp\Scal )^{\perp_1} = ({}^\perp\Scal )^{\perp}$, see \cite[2.2.11]{GT}.
\medskip

The following   Theorem, relying on work of Bazzoni, Herbera, and \v S\v tov\' \i\v cek, is essential for our investigation.

\begin{thm}\label{resol}
 Let $R$ be a ring, and let $(\mathcal A,\mathcal B)$ be a cotorsion pair. The following statements hold true.
 \begin{enumerate}
\item 
$(\mathcal A,\mathcal B)$ is an $n$-tilting cotorsion pair if and only if it is generated by a resolving subcategory  
 $\Scal$ of \m\ consisting of   modules of projective dimension at most $n$.
 \item 
 $(\mathcal A,\mathcal B)$ is an $n$-cotilting cotorsion pair if and only if it is cogenerated by a coresolving subcategory  
 $\Scal$ of ${\mathcal P \mathcal I}$ 
  consisting of   modules of injective dimension at most $n$ such that ${}^{\perp_1}\Scal$ is   closed under direct products.
\end{enumerate}
In particular,  tilting and cotilting classes are always {\em definable} classes, that is, they are closed under direct products, direct limits, and pure submodules.
\end{thm}
\begin{proof}
For (1) see \cite{bazher, bast} or \cite[5.2.23]{GT}.

(2) For the if-part, note that $(\mathcal A,\mathcal B)$ is a complete cotorsion pair
by \cite[3.2.9]{GT}. Further, from the  assumption on $\Scal$ it follows that $\Acal$ is closed under direct products and kernels of epimorphisms \cite[2.2.11]{GT} and  $\mathcal B$ consists of   modules of injective dimension at most $n$, see \cite[2.2]{AC1}. Then  $(\mathcal A,\mathcal B)$ is an $n$-cotilting cotorsion pair by \cite[4.2]{AC1}. 

For the converse implication, we use that every cotilting module is pure-injective, see \cite{B,S}, or  \cite[8.1.7]{GT}. So, if $C$ is a cotilting module such that ${}^\perp C=\mathcal A$, 
and $\Scal$ consists of the pure-injective modules in $\mathcal B$, then $\Scal$ is a coresolving subcategory of ${\mathcal P \mathcal I}$   that contains $C$ and has the stated properties, see \cite[8.1.10]{GT}. 
%Furthermore, by \cite[2.2.11]{GT}  
%$\mathcal A\subseteq {}^{\perp_1} \mathcal S ={}^{\perp} \mathcal S \subseteq {}^\perp C=\mathcal A$. 
Hence  $(\mathcal A,\mathcal B)$ is cogenerated by $\Scal$.
\end{proof}

Notice that a resolving subcategory $\Scal$ of \m\  as above is uniquely determined by the tilting class $\mathcal B$. Indeed, we have the following result.

\begin{thm}\cite[2.2 and 2.3]{AHT}\label{oneone}
There is a one-to-one correspondence between the $n$-tilting classes in $\M$ and the $n$-cotilting classes in $\LM$ that are cogenerated by a set $\Scal^\ast$ of dual modules. The correspondence is given by the assignment 
$$\Scal^{\perp} \mapsto {}^{\perp}\Scal^\ast (= \mathcal S ^\intercal)$$ 
where $\Scal$ is a  resolving subcategory  of \m\ as in Theorem \ref{resol}. Moreover, if $T$ is an  $n$-tilting right module, the dual module  $T^\ast$ is an  $n$-cotilting left module inducing the  corresponding cotilting class.
\end{thm}

We remark that if $R$ is left artinian, the assignment in Theorem \ref{oneone} even yields a one-to-one correspondence between the $1$-tilting classes in $\M$ and the $1$-cotilting classes in $\LM$; however, this need not be the case for general rings, cf.\ \cite[8.2.8 and 8.2.13]{GT}.

\bigskip

\section{Finiteness conditions on large tilting modules}

Throughout this section, $R$ denotes a ring and $\Scal$ a resolving subcategory of \m\ consisting of modules of projective dimension at most $n$. According to the results discussed in Section 1, the class
$\mathcal S$ gives rise to the following cotorsion pairs:

\subsection{The cotorsion pair $(\Mcal, \Lcal)$ generated by $\Scal$ in \M,} 
 \label{Lcal}
which is a tilting cotorsion pair by Theorem \ref{resol}. We fix an $n$-tilting module $T_R$ such that $\Lcal=T^\perp (= \mathcal S ^\perp)$.

\subsection{The cotorsion pair $(\Ccal, \Ccal^\perp)$ cogenerated by $\Scal^\ast$ in \LM,}   \label{Ccal} 
which is the  cotilting cotorsion pair corresponding to $(\Mcal, \Lcal)$ under the bijection of  Theorem \ref{oneone}. The  dual module  $T^\ast$ is an  $n$-cotilting left module such that $\Ccal={}^\perp (T^\ast) = T^\intercal = \mathcal S ^\intercal$.

If $R$ is left noetherian and $n=1$, then obviously $\Ccal=\varinjlim \Ccal^{<\omega}$.

\subsection{The cotorsion pair $({}^\perp\Dcal,\Dcal)$ generated by $\Ccal^{<\omega}$ in \LM.} 
\label{Dcal}
Of course, we have $ \Ccal^\perp\subseteq \Dcal$, and $\Ccal^{<\omega}$ is a resolving subcategory in $R$-mod.
Moreover, if $\Ccal^{<\omega}$ consists of modules of projective dimension at most $n$, then $({}^\perp\Dcal,\Dcal)$  is a tilting cotorsion pair by Theorem \ref{resol}, and we can fix an $n$-tilting module $_RW$ such that $\Dcal=W^\perp (= (\Ccal^{<\omega})^\perp)$.

\subsection{The cotorsion pair $(\mathcal F, \Ecal)$ cogenerated by $\Ccal^\ast$ in \M.}
\label{limScal}
This is the closure of the cotorsion pair $(\Mcal, \Lcal)$ studied in \cite{at2}.
Indeed, by the well-known  Ext-Tor-relations, we see that $\Ccal^\ast$   coincides with the class of all dual modules in $\Lcal$, cf.\ \cite[9.4]{relml}. 
 So $(\mathcal F, \Ecal)$ is cogenerated by the class of all dual modules in $\Lcal$, and 
 $\mathcal F =  \varinjlim\Mcal= \varinjlim\Scal = {}^\intercal (\mathcal S ^\intercal) = {}^\intercal \Ccal$ by \cite[2.1 and 2.3]{at2}.

Assume now that $\Ccal=\varinjlim \Ccal^{<\omega}$ and that $\Ccal^{<\omega}$ consists of modules of projective dimension at most $n$. Then 
$\mathcal F = {}^\intercal (\Ccal^{<\omega}) = {}^{\perp}((\Ccal^{<\omega})^\ast)$. Hence $(\mathcal F, \Ecal)$ is the cotilting   cotorsion pair corresponding to $({}^\perp\Dcal, \Dcal)$ under the bijection of  Theorem \ref{oneone}, and the  dual module  $W^\ast$ is an  $n$-cotilting right module such that $\mathcal F={}^\perp (W^\ast) = {}^\intercal W$.
%In fact, since we know from \ref{Ccal} that $\Ccal=\varinjlim\Ccal^{<\omega}$, we have 
% that $\varinjlim\Scal$ consists of the modules $X_R$ with $\Tor{R}{X}{C}=0$ for all $C\in \Ccal^{<\omega}$, which, again by the well-known  %Ext-Tor-relations,  gives that $\varinjlim\Scal={}^{\perp_1}((\Ccal^{<\omega})^\ast)$ is the cotilting class corresponding to $\Dcal$. 
In particular, note that in this case $\mathcal F$ is closed under direct products,  so \cite[4.2]{CB} implies that  $\Scal$ is covariantly finite in \m. 

\medskip

We now collect some characterizations of the case when $\Mcal$ is closed under direct limits.

\begin{thm}\label{resume}
 Assume  that  $\Ccal=\varinjlim \Ccal^{<\omega}$, and that $\Ccal^{<\omega}$ consists of modules of projective dimension at most $n$. The following statements are equivalent.
 \begin{enumerate}
 \item $\Mcal$ is closed under direct limits, that is, it coincides with  $\Fcal$.
\item $\Mcal$ is closed under direct products.
\item The cotorsion pair $(\Fcal, \Ecal)$ is generated by some subcategory of \m.
 \item $T$ is product-complete.
\end{enumerate}
If $R$ is right  noetherian, or $\Fcal$ consists of modules of bounded projective dimension, then (1) - (4) are further equivalent to
 \begin{enumerate}
\item[(5)] $W$ is endonoetherian.
\end{enumerate}
 \end{thm}
 \begin{proof}
 Of course, (1) implies (2) since $\Mcal$ is then a cotilting class by \ref{limScal}.
So,  the equivalence of the first four conditions follows immediately from \cite[2.3 and 3.1]{tel}. Moreover, 
under the additional assumptions, 
 condition (3) means that the class $\Ecal$ is closed under direct sums, 
see  \cite[4.10]{tel}, or \cite[5.1.16]{GT}. By \cite[3.2]{tel}, the latter 
  is equivalent to $W^\ast$ being $\Sigma$-pure-injective. Now we use \cite[9.9]{relml}, where it is shown that a tilting module is endonoetherian if and only if its dual module is $\Sigma$-pure-injective, and we obtain the equivalence of (3) and (5).
 \end{proof}
 
Symmetrically, we obtain

\begin{thm}\label{resume'}
 Assume that $\Ccal=\varinjlim \Ccal^{<\omega}$, and that $\Ccal^{<\omega}$ consists of modules of projective dimension at most $n$.
The following statements are equivalent:
 \begin{enumerate}
\item ${}^\perp\Dcal$ is closed under direct limits, that is, it coincides with $\Ccal$.
\item ${}^\perp\Dcal$ is closed under direct products.
\item The cotorsion pair $(\Ccal,\mathcal C^\perp)$ is generated by
 some subcategory of \lm.
 \item $W$ is product-complete.
\end{enumerate}
 If $R$ is left  noetherian, or $\Ccal$ consists of modules of bounded projective dimension, then (1) - (5) are further equivalent to
 \begin{enumerate}
\item[(5)] $T$ is endonoetherian.
\end{enumerate}
 \end{thm} 
 \begin{proof}
Consider the cotorsion pair  $({}^\perp \mathcal D,\mathcal D)$  generated by the resolving subcategory $\Ccal^{<\omega}$ of \lm\ consisting of modules of projective dimension at most $n$. Note that $\Ccal=\varinjlim\Ccal^{<\omega}=\varinjlim{}^\perp\Dcal$ by \cite[2.3]{at2}.
Moreover, recall that $\Fcal=\varinjlim\Scal$ is the cotilting class corresponding to $\Dcal$, and $\Fcal^{<\omega}=\Scal$. So,
  we can   apply Theorem \ref{resume}, keeping in mind that   the roles of $T$ and $W$ are now switched.
\end{proof}
 
\begin{cor}\label{endofin}
Let $R$  be a noetherian ring (or a right artinian ring, or a right  noetherian ring  of finite global dimension). Assume  that $\Ccal=\varinjlim \Ccal^{<\omega}$,  and that $\Ccal^{<\omega}$ consists of modules of projective dimension at most $n$. Then the following statements are equivalent.
\begin{enumerate}
\item $T$ is endofinite 
\item $\Mcal$ coincides with   $\Fcal$, and ${}^\perp\Dcal$ coincides with $\Ccal$.
\item $W$ is endofinite.
\end{enumerate} 
 \end{cor}
 \begin{proof}
% Since $\Ccal=\varinjlim\Ccal^{<\omega}$ and $({}^\perp \mathcal D,\mathcal D)$ is generated by $\Ccal^{<\omega}$, 
% we can apply Theorem \ref{resume} on the resolving subcategory $\Ccal^{<\omega}$ of \lm. 
%Recall that $\varinjlim\Scal$ is the cotilting class corresponding to $\Dcal$, and $(\varinjlim\Scal)^{<\omega}=\Scal$. 
%So, the roles of $T$ and $W$ are switched, and we obtain that ${}^\perp\Dcal$ coincides with $\Ccal$ if and only if $T$ is endonoetherian. 
Combine Theorems \ref{resume} and \ref{resume'}, and use that a module is endofinite if and only if it is endonoetherian and product complete (see e.g.\ \cite[p.43]{C2}). 
 \end{proof}
 
 \begin{cor}\label{contra}
 Assume that $R$ is an Artin algebra, $\Ccal=\varinjlim \Ccal^{<\omega}$, and $\Ccal^{<\omega}$ consists of modules of projective dimension at most $n$. Then the following statements are equivalent.
 \begin{enumerate}
\item $\Scal$ is contravariantly finite in \m.
\item $T$ can be chosen finitely generated.
\item $W$ can be chosen finitely generated.
\item $\Ccal^{<\omega}$ is contravariantly finite in \lm.
\end{enumerate}
In particular, in this case, we have cotorsion pairs     $(\mathcal F,\Lcal)$ in \M, and $(\Ccal,\Dcal)$ in \LM, where $\Lcal=\varinjlim\Lcal^{<\omega}$ and $\Dcal=\varinjlim\Dcal^{<\omega}$. 
\end{cor}
\begin{proof}
We apply some results from  \cite{AR}.
For the equivalence of (1)$\Leftrightarrow$(2) and (3)$\Leftrightarrow$(4), we refer to \cite[4.1]{at1}. Moreover, (1) implies that $D(\Scal)$ is coresolving and covariantly finite, hence (1)$\Rightarrow$(4) holds true by \cite[p.125]{AR}. Now assume that $W$ is finitely generated. Then  $\varinjlim \Ccal^{<\omega}={}^\perp\Dcal$ by  \cite[2.4]{KS}. Moreover,   $W$ and $D(W)$ are  endofinite, and $T$ is endofinite by Corollary \ref{endofin}.  Since $D(T)$ is a cotilting module with ${}^\perp D(T)=\Ccal$, and $W$ is a tilting module with $W^\perp=\Dcal$, we obtain $\Prod D(T) = \Ccal\cap\Dcal=\Add W$, see \cite[2.4]{AC1}. It follows that $D^2(T)\in\Prod D(W)=\Add D(W)$, and since the endofinite module $T$ is a direct summand in $D^2(T)$, and $D(W)$ is finitely generated, we deduce that $T$ is equivalent to a finitely generated tilting module. So we have verified (3)$\Rightarrow$(2).

The last claim is shown in \cite[2.4]{KS}, cf.\cite[5.3]{tel}.
\end{proof}

%\medskip

%\begin{rem}\label{criteria}
%{\rm
%If $W$ is finitely presented and $\Sigma$-pure-injective, then as in \cite[10.2]{relml} we have that $\Mcal$ coincides with its limit closure $\varinjlim\Scal$ if and only if every (countable) direct system in $\Scal$ is $T$-stationary (for the terminology, we refer to \cite{relml}).
%}
%\end{rem}

\bigskip

\section{Torsion pairs and Auslander-Reiten components}

Let us now consider  the case of $n=1$, that is, let $\Scal$ be a resolving subcategory of \m\ consisting of   modules of projective dimension at most one. From the  classes $\Lcal = \Scal^\perp$, $\Ccal = \Scal^\intercal$, $\Dcal = (\Ccal^{<\omega})^\perp$ and 
$\mathcal F = \varinjlim\Scal$ of the previous section  we also obtain some interesting torsion pairs. 

First of all, we have the tilting torsion pair $(T^o, \Lcal)$ in \M\ where $\Lcal=\Gen T = T^\perp$, and the cotilting torsion pair   
 $(\Ccal, {}^o(T^\ast))$ in \LM\ where $\Ccal=\Cogen T^\ast$. 

 Moreover, if  $\Ccal^{<\omega}$ consists of modules of projective dimension at most one, we also have the tilting torsion pair $(W^o, \Dcal)$ in \LM\ where $\Dcal=(\Ccal^{<\omega})^\perp = \Gen W$. 

 Finally, if $R$ is also left noetherian, then $\Ccal=\varinjlim \Ccal^{<\omega}$, and we have the cotilting  torsion pair 
 $(\Fcal,\, {}^o(W^\ast))$ in \M.

 Let us look at the last torsion pair in more detail. 
As we are going to see, if we assume the existence of almost split sequences, and take for $\Scal$ a union of Auslander-Reiten-components, then the torsion class coincides with  $\Gen\Lcal^{<\omega}$.

\medskip

  {\bf Definition.}  \cite{AV} 
Let $R$ be a right artinian ring. 
 A subcategory $\comp$ of \m\ consisting of indecomposable modules   is called an  {\em Auslander-Reiten component in}  \m\ if it satisfies the following conditions.
\begin{enumerate}
\item For any $X\in\comp$ there are a left almost split morphism $X\ra Z$
%, if $X$ is not simple injective,
 and a right almost split morphism $Y\ra X$
%, provided $X$ is not simple projective, 
in \M\ with $Z,Y\in\m$.
\item If $X\ra Y$ is an irreducible map
 in \m\ 
with one of the modules lying in  $\comp$, then both modules are in  $\comp$.
\item  The Auslander-Reiten-quiver of $\comp$ is connected.
\end{enumerate}

\medskip

The next result and its dual  allow us to replace the Auslander-Reiten-formula \cite{C2} when dealing with artinian rings that need not have a self-duality.

\begin{lemma}\label{3.6}  \cite[3.6]{AV}
Let $0\to A\to B\to C\to 0$ be an almost split sequence in \M, where $R$ is an arbitrary ring, and $X_R$ a module. If $ \mbox{\rm Hom}_R\,(X,A)=0$, then also $\Ext{R}{C}{X}=0$.
The converse holds if $C$ has projective dimension at most one.
\end{lemma}

\begin{cor}\cite[1.4]{key}\label{ARformula}
Let $R$ be a  hereditary ring, and let $\exs{A}{}{B}{}{C}$ be an almost split sequence in \M. Then $^\circ A=C^\perp$ {and} $C^\circ ={^\perp A}$.
\end{cor}

 \begin{prop}\label{torsion}
Let    $\Scal$ be a resolving subcategory of \m\ consisting of   modules of projective dimension at most one, and set $\Lcal=\Scal^\perp$.
 \begin{enumerate}
 \item Assume $R$ is right noetherian.
Then there is a torsion pair 
  $(\varinjlim (T^o)^{<\omega}, \Gcal)$   in \M\ with the torsion-free class  $\varinjlim (T^o)^{<\omega}=(\Lcal^{<\omega})^o$ and the torsion class $\Gcal= \Gen \Lcal^{<\omega}= \varinjlim\Lcal^{<\omega}$. 
 \item 
 Assume    that $R$ is  artinian and hereditary, and that every finitely generated indecomposable non-injective right $R$-module is first term of an almost split sequence in \M\ consisting of finitely generated modules. Assume further that the  indecomposable modules of $\Scal$ are not injective and form  a union of Auslander-Reiten-components. 
Then $\varinjlim (T^o)^{<\omega} = \mathcal F$, and  there is a torsion pair  $(\mathcal F, \Gen \Lcal^{<\omega})$   in \M.
\end{enumerate}   
\end{prop}
\begin{proof}
(1) Since  $R$ is right noetherian, we know from \cite[4.5.2]{GT} that $(T^o, \Lcal )$ induces a torsion pair $((T^o)^{<\omega}, \Lcal^{<\omega})$ in \m. Now we apply \cite[4.4]{CB} to obtain a torsion pair $(\varinjlim(T^o)^{<\omega}, \varinjlim\Lcal^{<\omega})$ in \M\ with the stated properties.
\\
(2) First, since $R$ is left noetherian and  $\Ccal^{<\omega}$ consists of modules of projective dimension at most one, we have the cotilting  torsion pair  $(\mathcal F, {}^o(W^\ast))$ in \M. 
 
 Furthermore, the assumptions on $\Scal$ imply by   Corollary \ref{ARformula} that $\Lcal=\Scal^\perp={}^o \Scal$. Hence $(T^o, \Lcal)$ coincides with  the torsion pair 
  $( ({}^o \Scal)^o, {}^o \Scal)$ generated by $\Scal$. 
  
  We claim that $(T^o)^{<\omega}=\Scal$. The inclusion $\supseteq$ is obvious. For the reverse one, consider $A\in (T^o)^{<\omega}\subseteq\Lcal^o$, and assume w.l.o.g.\ that $A$ is indecomposable. Since the tilting class $\Lcal$ generates all the injective modules, $A$ is not injective. Thus there is an almost split sequence $0\to A\to B\to C\to 0$
  in \M\ consisting of finitely generated modules, and for all $X\in\Lcal$ we have   by Corollary \ref{ARformula}
  that $X\in{}^oA=C^\perp$, which shows that $C\in\Mcal$. Then $C\in\Mcal^{<\omega}=\Scal$ by \cite[5.2.1]{GT}, hence also $A\in\Scal$. 
  
  Finally, by (1), $\Fcal=\varinjlim\Scal = \varinjlim (T^o)^{<\omega}$, which concludes the proof.
\end{proof} 

%\begin{rem}\label{pure}
%In the situation of Proposition \ref{torsion}(2), every exact sequence $\varepsilon: 0\to X\to Y\to Z\to 0$ with $X\in\Gen\Lcal^{<\omega}$ and  %$Z\in\varinjlim\Scal$ is pure-exact.
%In fact, there is a pure-epimorphism   $g:\bigoplus_{i\in I} M_i\to Z$ for some modules $M_i\in\Scal$, and $g$ factors through $\varepsilon$ %since $X\in\Gen\Lcal^{<\omega}\subset \Lcal= \Scal^\perp$ implies that $\Ext{R}{{\bigoplus_{i\in I} M_i}}{X}=0$. 
%\end{rem}

\begin{ex}\label{needed}
{\rm
Assume that $R$   is an  indecomposable, artinian, hereditary, left pure-semisimple ring of infinite representation type. Choose $\Scal=\add\p$ where $\p$ is the preprojective component of \m, and consider the   corresponding torsion and cotorsion pairs.  

It is shown in \cite[4.3]{key} that in  this case  the module $_RW$ from section \ref{Dcal}   is finitely generated and product-complete, and the  classes ${}^\perp\Dcal$ and $\Ccal$ in  \ref{Dcal} and \ref{Ccal} coincide, cf. Theorem \ref{resume}. In other words,  we have a (co)tilting cotorsion pair $(\Ccal, \Dcal)$ in \LM\  generated by $\Ccal^{<\omega}$,  and $W$ is a (co)tilting module (co)generating this cotorsion pair. 

Moreover, in \LM\ we have a split cotilting torsion pair $(\Ccal, {}^o(T^\ast))$, and a split tilting torsion pair $(W^o,\Dcal)$, where $W^o\subseteq \Ccal$, see \cite[4.6 and 4.7]{key}.
In \M\ we have a tilting torsion pair $(T^o,\Lcal)$, and  by Proposition \ref{torsion}(2), we have a cotilting torsion pair $(\varinjlim\add\p, \Gen \Lcal^{<\omega})$.

Now, since  we are assuming that $R$ has infinite representation type,  we are also assuming that there are finitely generated indecomposable non-injective left $R$-modules which are not first terms of an almost split sequence in \lm, see for example \cite{ZH1}. So the assumptions of Proposition \ref{torsion}(2) are not satisfied in $\LM$.  Indeed, 
let us consider the tilting torsion  pair $(W^o,\Dcal)$ in \LM. By Proposition \ref{torsion}(1) we  obtain a torsion pair $(\varinjlim (W^o)^{<\omega}, \tilde{\Gcal})$ in \LM.
  Here the  torsion-free class $\varinjlim (W^o)^{<\omega}=(\Dcal^{<\omega})^o$ is  contained in the cotilting torsion-free class $\Ccal=\varinjlim \Ccal^{<\omega}$, but in contrast to Proposition \ref{torsion}(2), this inclusion is proper. In fact, $W$ belongs to $\Ccal$, but it does not belong to $(\Dcal^{<\omega})^o$ as $W\in\Dcal^{<\omega}$. 
}\end{ex}

\bigskip

\section{Relative Baer modules}

Let $R$ be an arbitrary ring and $\mathcal T$ be a torsion class in $\rmod
R$. A module $M \in \rmod R$   is a {\em Baer module for $\mathcal T$} provided that $\Ext RMT = 0$ for all $T \in \mathcal T$.

It is shown in \cite{baerml} that the study of Baer modules can often be reduced to the countably generated case. To recall this, we need further notation.

Let $\sigma$ be an ordinal. An increasing chain of submodules,
$\mathcal J = ( M_\alpha \mid \alpha \leq \sigma )$, of a module
$M$ is called a {\em filtration} of $M$ provided that $M_0 = 0$,
$M_\alpha = \bigcup_{\beta < \alpha} M_\beta$ for all limit
ordinals $\alpha \leq \sigma$ and $M_\sigma = M$.

Given a class of modules $\mathcal C$ and a module $M$, a filtration $\mathcal J$ is a {\em $\mathcal
C$--filtration} of $M$ provided that $M_{\alpha + 1}/M_\alpha$ is isomorphic to some element of
$\mathcal C$ for each $\alpha < \sigma$. In this case we say that $M$ is {\em $\mathcal C$--filtered}.

%Given a cardinal $\kappa$, a filtration $\mathcal M$ is a {\em $\kappa$--filtration} of $M$
%provided that $M_{\alpha}$ is $< \kappa$--generated for each $\alpha < \sigma$.

%
\begin{thm}\cite[Theorem 1]{baerml}\label{reduction}
Let $R$ be an $\aleph_0$--noetherian ring and $\mathcal T$ be a torsion class in $\rmod R$ 
%closed under direct sums and 
such that $^{\perp_1} \mathcal T = {}^\perp \mathcal T$. Assume that either $\mathcal T$ consists of
modules of finite injective dimension, or ${}^\perp \mathcal T$ consists of modules of finite
projective dimension. Then a module $M$ is a Baer module  for $\mathcal T$
if and only if it has a filtration $\mathcal M = ( M_\alpha \mid \alpha \leq \kappa )$ such that,
for each $\alpha < \kappa$,
$M_{\alpha + 1}/M_\alpha$ is a countably generated Baer module for $\mathcal T$.
\end{thm}
 
\medskip
 
Assume again that $\Scal$ is a resolving subcategory of \m\ consisting of modules of projective dimension at most one, and let $(\Mcal,\Lcal)$ be the cotorsion pair generated by $\Scal$. Notice that $\mathcal M = {}^{\perp_1} \Lcal = {}^\perp \Lcal $ consists of modules of projective dimension $\leq 1$, see \cite[2.2]{AC1}.

We will now focus on Baer modules for the torsion class $\Gcal = \Gen \Lcal^{<\omega}$ from Proposition \ref{torsion}. We will denote this class by $\mathcal B$.
 
 \begin{lem}\label{generalbaerm}
 Let $R$ be a right noetherian ring. Assume that every module $A\in\Lcal$ can be purely embedded in a direct product of modules from $\Lcal^{<\omega}$, and that $\mathcal B$ consists of modules of projective dimension $\leq 1$.
  Then $\mathcal B = \Mcal$. 
 \end{lem}
  \begin{proof}
Since $\Gen\Lcal^{<\omega}\subseteq{\mathcal L}$, we have  $\Mcal={}^\perp {\mathcal L}\subseteq \mathcal B$. For the reverse inclusion, let $M \in \mathcal B$. By 
%our assumption on $\Gcal$, $\mathcal B$ consists of modules of projective dimension $\leq 1$, so by 
Theorem \ref{reduction} and by the Eklof Lemma \cite[3.1.2]{GT}, we can assume w.l.o.g.\ that $M$ is countably generated. Since $M^{\perp_1}$ contains $\Gcal$ and   $\Gcal$ is closed under direct sums, we conclude from \cite[2.5]{bazher} or \cite[2.7]{SarSt} that $M^{\perp_1}$ also contains every pure submodule of a direct product of modules in $\Gcal$. So by our assumption, $M^{\perp_1}$ contains $\Lcal$, which obviously means that $M\in\Mcal$.
\end{proof}

Since $(\mathcal M, \mathcal L)$ is generated by $\mathcal S$, $\mathcal M$ coincides with the class of all direct summands of $\mathcal S$-filtered modules (cf.\ \cite[3.2.4]{GT}). We have a stronger result in the particular case of artin algebras:  

\begin{thm}\label{filtbaer}
 Let $R$ be an artin algebra. Then $\mathcal B = \mathcal M$ coincides with the class of all $\mathcal S$-filtered modules.
 \end{thm}
  \begin{proof}
First, every module can be purely embedded in the product of all its finitely generated factor modules, see \cite[2.2.Ex 3]{C2}. Of course, if $A$ belongs to $\Lcal$ then all its finitely generated factor modules do. Since $\Lcal$ contains all injective modules, and each indecomposable injective module is finitely generated,
$\Gcal$ contains all homomorphic images of injective modules. This implies that  $\mathcal B$ consists of modules of projective dimension $\leq 1$. Hence
 Lemma \ref{generalbaerm} applies and gives $\mathcal B = \mathcal M$. Moreover, the Eklof Lemma \cite[3.1.2]{GT} gives that each $\mathcal S$-filtered module is in $\mathcal M$. 

Conversely, let $M \in \mathcal M$. By Theorem \ref{reduction}, $M$ is $\mathcal N$-filtered where $\mathcal N$ is the class of all countably generated modules from $\mathcal M$. So it remains to prove that each countably generated module $M \in \mathcal M$ is $\mathcal S$-filtered.

By \cite[3.2.4]{GT}, there is a module $N$ which is a union of an $\mathcal S$-filtration $( N_i \mid i \leq \sigma)$ such that $M$ is a direct summand in $N$. By the Hill Lemma \cite[4.2.6]{GT}, we can w.l.o.g.\ assume that $\sigma = \omega$. 
 
By induction on $i < \omega$, we will construct an $\mathcal S$--filtration, $(M_i \mid i \leq \omega )$, of the module $M$. Let $\{ g_i \mid i < \omega \}$ be an $R$--generating subset of $M$. Denote by $t$ the torsion radical corresponding to the torsion pair $(\mathcal F,{}^o\mathcal F)$ (where 
$\mathcal F = \varinjlim \mathcal S$ is the  torsion-free class from \ref{limScal}). Put $M_0 = 0$, and if $M_i$ is defined so that $M_i$ is finitely generated and $t(M/M_i) = 0$, we consider the least index $j < \omega$ such that $\left\langle  M_i,g_i \right\rangle \subseteq N_j$, and let $M_{i+1} = M \cap N_j$. 

Then $M_{i+1}$ is finitely generated because $N_j$ is such, and moreover $M/M_{i+1} \cong (M + N_j)/N_j \subseteq N/N_j$, so $t(M/M_{i+1}) = 0 = t(N/N_j)$ because $N/N_j$ is $\mathcal S$-filtered. 

Since $g_i \in M_{i+1}$ for each $i < \omega$, we have $M = M_\omega = \bigcup_{i < \omega} M_i$. 
Finally $\mathcal F$ is a resolving class, so the exact sequence 
$0 \to M_{i+1}/M_i \to M/M_i \to M/M_{i+1} \to 0$ yields $M_{i+1}/M_i \in \mathcal F ^{<\omega} = \mathcal S$,       
proving that $(M_i \mid i \leq \omega )$ is an $\mathcal S$-filtration of $M$. 
\end{proof}                     

\begin{rem} Theorem \ref{filtbaer} was first proved in the particular case when $R$ is a tame hereditary artin algebra and $\mathcal S$ is the class of all preprojective modules, see \cite{baerml}. There are many more analogies with \cite{baerml}: 

Let $R$ be a hereditary artin algebra, and $\Scal$ is a resolving subcategory of \m\,
%consisting of modules of projective dimension at most one 
satisfying the assumptions of Proposition \ref{torsion}(2). Denote by $\ell$ the torsion radical corresponding to the torsion pair $(\mathcal L ^o,\mathcal L)$. By Proposition \ref{torsion}(2) we have $\mathcal L ^o \subseteq \mathcal F$. 

Let $\mathcal B$ be the class of all Baer modules for $\mathcal G=\Gen \Lcal^{<\omega}$. Two modules $B, B^\prime \in \mathcal B$ are called {\em equivalent} iff $B/\ell(B) \cong B^\prime/\ell(B^\prime)$.     

As in \cite{baerml}, one can prove in our general setting that 

(1) $\mathcal F$ is exactly the class of all pure epimorphic images of the modules in $\mathcal B$.

(2) If $M \in \mathcal F$ then $\ell(M)$ is a pure submodule of $M$; if moreover $M \in \mathcal B$ then $\ell(B) \in \mbox{Add}(T)$.

(3) If  $B, B^\prime \in \mathcal B$, then $B$ is equivalent to $B^\prime$ iff there exist $L, L^\prime \in \mbox{Add}(T)$ such that $B \oplus L \cong B^\prime \oplus L^\prime$.   

(4) Equivalence classes of modules in $\mathcal B$ correspond bijectively to isomorphism classes of modules in the torsion--free class $\mathcal L ^o$.

However, there does not appear to be any general decomposition theorem for countably generated Baer modules extending \cite[Prop.13]{baerml} and \cite[4.3]{L2}.
\end{rem}

%\begin{rem}
%{\rm It is well known that over an Artin algebra every module can be purely embedded  in the product of all its finitely generated factor modules. Of %course, if $A$ belongs to $\Lcal$, its finitely generated factor modules are again  in $\mathcal L$, so the assumption in the Theorem above is always %satisfied. }
%\end{rem}
 
As a consequence of our description of the Baer modules, we obtain yet another characterization of when $\Mcal$ is closed under direct limits.

 \begin{cor}\label{splittp}
% Assume    that $R$ is  artinian and hereditary, and that every finitely generated indecomposable non-injective right $R$-module is first term of an almost split sequence in \M\ consisting of finitely generated modules. 
Let $R$ be a hereditary artin algebra.
 Assume further that the  indecomposable modules of $\Scal$ are not injective and form  a union of Auslander-Reiten-components.
%, and that every module $A\in\Lcal$ can be purely embedded in a direct product of finitely generated modules from $\Lcal$. 
Then $\Mcal$ coincides with $\mathcal F$ if and only if  $(\mathcal F,\Gen \Lcal^{<\omega})$ is a split  torsion pair.
  \end{cor}
  \begin{proof}
We know from Proposition \ref{torsion}(2) that $(\mathcal F,\Gen \Lcal^{<\omega})$ is a   torsion pair. So, we have to show that  $\Mcal=\mathcal F$ if and only if $\Ext{R}{Z}{X}=0$ for all $X\in \Gen \Lcal^{<\omega}$ and $Z\in\mathcal F$.

The only-if part follows immediately from the fact that  $\Gen \Lcal^{<\omega}\subseteq \Lcal=\Mcal^\perp$. For the if-part, observe that if $\mathcal F \subseteq{}^{\perp_1} (\Gen \Lcal^{<\omega})$, then $\mathcal F$ consists of Baer modules for $\Gen \Lcal^{<\omega}$, hence it is contained in $\Mcal$ by Theorem \ref{filtbaer}.
\end{proof}
 
%\begin{rem}\label{lukascond} 
%{\rm
%Let the assumptions and notations be as in Corollary \ref{splittp}.
%Denote by  $t$  the
% torsion radical corresponding to the torsion pair $(\varinjlim\Scal,\Gen\Lcal^{<\omega})$.  
%Let $M\in\varinjlim\Scal$ be a countably generated module.
% Then $M$ belongs to $\Mcal$ if and only if
% the following condition
%holds true (cf. \cite[4.1]{L2}):
%\begin{quote} If $N$ is a finitely generated submodule of $M$, then $t(M/N)$ is finitely generated.
%  \end{quote}
%In fact, for the if-part, one shows as in \cite[10.3]{relml} that $M$ is $W$-Mittag-Leffler, 
% which means $M\in\Mcal$ by \cite[9.8]{relml}. For the %only-if part see \cite[4.1]{L2}.
%}
%\end{rem}

\bigskip

\section{Applications to hereditary artin algebras}

>From now on, we will assume that $R$ is an indecomposable representation-infinite hereditary artin algebra with the standard duality $D: \m \to \lm$.
%(Actually, for most results it is enough to assume that $R$ is an artinian indecomposable hereditary ring admitting a preprojective component in \m, see \cite{AV}, \cite{key}.)

Let $\q$, $\tube$ and $\p$   denote representative sets of
all indecomposable finitely generated preinjective, regular, and preprojective right modules, respectively. The corresponding sets of left modules are denoted by $_R\q$, $_R\tube$ and $_R\p$.  

We now apply our previous considerations to the resolving category $\Scal=\add \p$. Then our torsion pairs look as follows.

\subsection{The torsion pair generated by $\p$  in \M.} \label{preprojective}
 By the  Auslander-Reiten
formula   $$\mathcal L = \p^\perp= {}^o\p $$ so   $\mathcal L$ is the class of all {right} modules having no non-zero homomorphism to $\p$, or in other words, the class of all modules that have no non-zero finitely generated preprojective
direct summands (see \cite[Corollary 2.2]{R}).
There is a countably infinitely generated tilting right module generating $\mathcal L$, called the {\em
Lukas tilting module}, and denoted by $L$, cf.\ \cite{KT}. 

%For a module $M$, we will denote by $\ell(M)$  the trace of $L$ in $M$.
The torsion--free class corresponding to $\mathcal L$ will be denoted by $\mathcal P$. This is the
class of all (possibly infinitely generated) {\em preprojective right modules}.
 We have
$$ \mathcal P \cap \rfmod R  = \add \p$$
 
%In particular, the preprojective indecomposable modules are precisely the copies of 
%modules from $\p$. Similarly, the finitely $\p$--filtered modules are exactly the modules in $\add\p$.   

\noindent
(Note: in \cite{L1} and
\cite{L2}, preprojective modules are called `$\mathcal P ^\infty$-torsion-free', and the modules in $\Lcal$ are called `$\mathcal P ^\infty$-torsion').

\subsection{The torsion pair generated by $_R\q$  in \LM.}  \label{preinjective}
 By the Auslander-Reiten formula    $$\mathcal C  = {}^\perp (_R\q)= {}_R\q^o$$
so 
$ \mathcal C = {}_R\q^o $ is the class   of all {left} modules having no non-zero homomorphism from $\q$, or in other words,   the class of all modules that have no  non-zero finitely generated preinjective
direct summands.     By Theorem \ref{oneone}, we know that $D(L)$ is a cotilting module cogenerating $\mathcal C$.
 
% For a module $M$, we will denote by $q(M)$  the trace of $_R\q$ in $M$.
 The corresponding torsion class is denoted by $\mathcal Q$.  This is the
class of all  (possibly infinitely generated) {\em preinjective left modules}, that is, of all (possibly infinite) direct sums of modules from $_R\q$, see \cite[3.3]{R}. 

Note that in the tame case, the torsion pair $(\mathcal C,\mathcal Q)$ is a split torsion pair, see \cite{R,RR}.

\medskip

\subsection{The torsion pair cogenerated by $_R\tube$  in \LM.}  \label{divisile}
As $\Ccal^{<\omega}=\add(_R\p\cup _R\tube)$, and since  from every module in $_R\p$ there is a non-zero map    to some module in  $_R\tube$,  
we infer from the Auslander-Reiten formula that
$$\mathcal D=(_R\tube)^\perp={}^o(_R\tube)$$
 is the torsion class of all {\em divisible  left modules}, see \cite{R}. The corresponding torsion--free class, called the class of all {\em reduced left modules}, is denoted by $\mathcal R$. 

We fix a tilting left module $W$ 
which generates $\mathcal D$.
If $R$ is tame, then it is shown in \cite{RR} that $W$
can be chosen as the direct sum of a set of representatives of the Pr\"ufer left $R$-modules and the generic left $R$-module $_RG$. This module is called the {\em Ringel tilting module}.  Moreover, in the tame case, the torsion pair $(\mathcal R,\mathcal D)$ is a split torsion pair, see \cite{R,RR}.

\subsection{The torsion pair generated by $\tube$ in   \M.}\label{torsionfree}
Dually we see that $\Lcal^{<\omega}=\add (\q\cup\tube)$, hence
by the Auslander-Reiten formula  $$\mathcal F = \tube^o = {}^\perp \tube$$
is the class of all {\em torsion-free right modules}, see \cite{R}. 
Moreover, $\Fcal=\varinjlim\add\p$, 
 and  $D(W)$ is a cotilting module which cogenerates $\mathcal F$, see Section \ref{limScal} and Proposition \ref{torsion}. In the tame case, $D(W)$ is the direct product of a set of representatives of the adic right $R$-modules and the generic right $R$-module $G_R$.
  
The corresponding torsion class is $\Gen\Lcal^{<\omega}=\Gen \tube$, called the class of all {\em torsion modules}, see \cite[3.5]{RR}. %For a module $M$, we  denote by $t(M)$  the trace of $\tube$ in $M$.

\medskip

Notice that $$\mathcal P = L^o= (^o \p)^o \subseteq \tube^o = \mathcal F, \quad\text{and}\quad
\mathcal Q ={}^o D(L)={}^o (_R\q^o) \subseteq {}^o(_R\tube) = \mathcal D$$

More precisely, we remark the following properties of the torsion pairs above. 

\subsection{Extremal torsion pairs.}\label{extremal}
The class $\mathcal{Q}=\, ^{\circ}(_R\mathbf{\q}^{\circ})$ is the smallest possible torsion class in \LM\ containing $_R\mathbf{q}$,  and
the class $\mathcal{R}=(\, ^{\circ}(_R\mathbf{t}))^{\circ}$ is the smallest possible torsion--free  class in \LM\ containing $_R\mathbf{t}$.
Note that $\mathcal{R}=W^o$ also contains $_R\p$.

So,
both torsion pairs $(\mathcal{C}, \mathcal{Q})$ and $(\mathcal{R}, \mathcal{D})$ have the property that the indecomposable finite length modules in the torsionfree class are precisely the modules  in $_R\p\cup{}_R\tube$, and the indecomposable finite length modules in the torsion class are precisely the modules in $_R\q$. Moreover, as shown in \cite[\S 3]{RR}, they are extremal with this property. More precisely, 
 if $(\mathcal{X},\mathcal{Y})$ is a  torsion pair in \LM\ such that 
  $\mathcal{X}$ contains $_R\mathbf{t}$ and $\mathcal{Y}$ contains $_R\mathbf{q}$,
   then 
$$\mathcal{R}\subseteq\mathcal{X}\subseteq\mathcal{C}\mbox{ and }
\mathcal{Q}\subseteq\mathcal{Y}\subseteq\mathcal{D}.$$

\medskip

Similarly, the class $\mathcal{P}=(^{\circ}\mathbf{p})^{\circ}$ is the smallest possible torsion--free class in \M\ that contains $\mathbf{p}$,
and
Gen$\tube$  is the smallest possible torsion class in \M\  containing $\mathbf{t}$.
Moreover, Gen$\tube$ also contains $\q$.

So, both torsion pairs $(\mathcal{P}, \mathcal{L})$ and $(\mathcal{F}, \text{Gen}\tube)$ have the property that the indecomposable finite length modules in the torsionfree class are precisely the modules  in $\p$, and the indecomposable finite length modules in the torsion class are precisely the modules in $\tube\cup\q$. Furthermore, they are extremal with this property, in the sense that if $(\mathcal{X},\mathcal{Y})$ is a  torsion pair in \M\ such that 
  $\mathcal{X}$ contains $\mathbf{p}$ and $\mathcal{Y}$ contains $\mathbf{t}$,
   then 
$$\mathcal{P}\subseteq\mathcal{X}\subseteq\mathcal{F}\mbox{ and }
\text{Gen}\tube\subseteq\mathcal{Y}\subseteq\mathcal{L}.$$

\medskip

\subsection{Baer modules}

A module $M$ is called {\em Baer} provided that $M$ is a Baer module for $\Gen \tube$.
As an application of Theorem \ref{filtbaer}, we obtain the following result (whose particular instance
for tame algebras is \cite[Theorem 2]{baerml}).

\begin{cor}\label{baerm}
Let $R$ be an indecomposable representation-infinite hereditary artin algebra. 
 A right $R$-module $M$ is Baer if and only if it is $\p$-filtered.
 \end{cor}
% \begin{proof}
% From Theorem \ref{generalbaerm} and Section \ref{torsionfree} we know that 
% the Baer modules are precisely the modules in the left class of the cotorsion pair $(\Mcal,\Lcal)$   generated by $\add\p$. But this class consists of %all direct summands 
%of $\p$-filtered modules by \cite[3.2.4]{GT}.
%\end{proof}

As in \cite{baerml} we obtain as  consequences
 
\begin{cor}\label{corol}  A module $M$ is Baer if and only if there is an exact sequence
$0 \to M \to L_1 \to L_2 \to 0$ where $L_1, L_2 \in \Add L$.
\end{cor}
\begin{cor}\label{corola}  The following statements are equivalent for a module $M$.
\begin{enumerate}
\item
 $M$ is torsion--free.
 \item
 $M$ is a pure-epimorphic image of direct sum of indecomposable preprojective modules.
 \item  $M$  occurs as the end term in a pure--exact sequence
$$0 \to N \to B \to M \to 0$$ with a Baer module $B$  
and  $N \in \Add L$.
\end{enumerate}
\end{cor}

\bigskip

\section{Representation type}
Again we assume that $R$ is an indecomposable  hereditary artin algebra.
We now   consider   the cotorsion pairs  $(\Ccal,\mathcal C^\perp)$ and $({}^\perp \mathcal D,\mathcal D)$ in $R$-Mod cogenerated by $_R\q$ and generated by $_R\tube$, respectively.
Note that they coincide when $R$ is of tame representation type,  as shown in \cite{RR}. In fact, this characterizes  the tame case.

\begin{thm}\label{tame}
The following statements are equivalent.
 \begin{enumerate}
 \item $(\Ccal,\Dcal)$ is a cotorsion pair.
 \item $W$ is product-complete.
 \item $L$ is endonoetherian.
 \item The torsion pair $(\Ccal, \Qcal)$ splits.
 \item $R$ is of tame representation type.
\end{enumerate}
 \end{thm}
 \begin{proof}
 The equivalence of the first three conditions is just Theorem \ref{resume'}. Moreover, (1) and (4) are equivalent by Corollary \ref{splittp}. Finally, the equivalence of (4) and (5) is shown by Ringel \cite[3.7 - 3.9]{R}.
 Alternatively, one can use Kerner's construction  over a wild hereditary algebra
 of an indecomposable divisible module in $\Ccal$ which does not belong to $^\perp \mathcal D$, see \cite[1.7 and p.416]{L1}. This shows that in the wild case $^\perp \mathcal D$ is properly contained in
 $\Ccal$ and therefore proves (1)$\Rightarrow$(5). The converse implication is proven in \cite{RR}.
 \end{proof}
 
% \begin{rem}
%{\bf
%This is how we had argued in Varese, however  using a result by  Lenzing that relies on Ringel's result.}{\rm
%We   show that (1) implies (8). If  $(\mathcal C,\mathcal D)$ is a cotorsion pair, then it is a tilting cotorsion pair generated by the tilting %module $W$. In particular, it follows that ${\mathcal C}\cap{\mathcal D}=\Add W$. Moreover, since $\mathcal C$ is definable, it follows from %\cite[3.1]{tel} that $W$ is product-complete, and therefore it is a direct sum of indecomposable $\Sigma$-pure-injective modules (with local %endomorphism ring). 

%Let now  $D\in {\mathcal D}$. By \cite[3.3]{R} there is a pure-exact sequence $$0\to q(D)\to D\to D/q(D)\to 0$$ 
%Then $q(D)$ belongs to ${\mathcal Q}\subset {\mathcal D}=W^\perp$, and  $D/q(D)$ belongs to ${\mathcal C}$, but also to the torsion class %${\mathcal D}$. Hence  $D/q(D)\in \Add W$, and the sequence splits. This shows that $D= q(D) \oplus D/q(D)$ is a direct sum of modules in $\q$ %and  indecomposable direct summands of $W$. 

%So, we conclude that    every module in $\mathcal D$ is a direct sum of indecomposable modules. 
%The latter means that $R$ is tame by \cite[4.9]{Le2}.
%}\end{rem}

 Now let us consider the relationship between the cotorsion pairs $(\Mcal, \Lcal)$ and
 $(\Fcal, \Ecal)$, as defined in 2.1 and 2.4. 
  
\begin{thm}\label{frt}
Assume that $R$ is of tame representation type. The following statements are equivalent.
 \begin{enumerate}
 \item $\Mcal$ is closed under direct limits, that is, it coincides with  $\Fcal$.
 \item $L$ is product-complete.
 \item $W$ is endonoetherian.
 \item $R$ is of finite representation type.
\end{enumerate}
 \end{thm}
 \begin{proof}
 The equivalence of the first three conditions is  just Theorem \ref{resume}. Moreover, if $R$ is a ring of finite representation type, then all modules are endofinite.  So, (4) trivially implies (3). Conversely, it is known that $\Mcal$ is properly contained in $\Fcal$ when $R$ is tame of infinite representation type. For example, the generic module is torsion-free but not Baer, see \cite{baerml}. Thus (1) implies (4).
 
% Alternative argument: By
 %  Corollary \ref{splittp},   condition
 %(1) means that the torsion pair $(\Fcal, \Gen\tube)$ splits, and 
 %it is well known that this cannot happen if $R$ is tame of  infinite representation type, see \cite{R,RR}. Let us check this.
 %Let $R$ be tame of infinite representation type. Choose a family of pairwise non-isomorphic indecomposable simple regular modules $(X_i)_{i\in %\N}$, and
% consider the pure-exact sequence $$0\to \bigoplus_{i\in \N} X_i\to\prod_{i\in \N} X_i\to Z\to 0$$ 
% By \cite[Proposition 5 and Remark]{RSpe} the module  $Z$ is a direct sum of copies of the generic module $G$, and the sequence cannot split. %This shows that the torsion pair $(\Fcal, \Gen\tube)$ does not split.
 \end{proof}

\begin{cor}\label{wild}
Assume that $R$ is of infinite representation type. Then $(\Fcal,\Gen\tube)$ does not 
split, and $\Mcal$ is properly contained in $\Fcal$.
\end{cor}
\begin{proof}
 By Corollary \ref{splittp} the torsion pair $(\Fcal,\Gen\tube)$   
splits if and only if $\Mcal=\Fcal$. But if this is the case, then we know from Theorem \ref{resume}  that $L$ is   product-complete.  In particular, $L$ is $\Sigma$-pure-injective, so it has a decomposition $L=\bigoplus L_i$ in indecomposable modules with local endomorphism ring. By \cite[6.1.b(ii)]{L1} every $L_i$ has the property that $\Add L\subseteq \Add L_i$. By the Theorem of Krull-Remak-Schmidt-Azumaya it follows that all $L_i$ are isomorphic, that is, 
$L$ is equivalent to an indecomposable product-complete, thus endofinite, tilting module. But 
by Corollary \ref{endofin} this implies that ${}^\perp\Dcal$ coincides with $\Ccal$, and $\Mcal$ coincides with $\Fcal$.  This means that $R$ is of finite representation type by Theorems \ref{tame} and \ref{frt}. 
\end{proof}

%{\bf Question:} 
%(1) In Theorem \ref{baerm}, can we even prove that a right $R$-module $M$ is  a {Baer  module} 
% if and only if it is  $\p$-filtered ?
 
%Is it true that $R$ is tame iff $(\Rcal,\Dcal)$ splits? See \cite[p. 415]{L1}.

\bigskip

% = Acknowledgement ===========================================================

% = Bibliography ==============================================================


\begin{thebibliography}{EFS}

%\bibitem{via}{\sc  L.~Angeleri H\"ugel}: Covers and envelopes via endoproperties of modules, Proc. London Math. Soc. {\bf 86} (2003) 649-665.

\bibitem{key}{\sc  L.~Angeleri H\"ugel}, {\sl A key module over pure-semisimple hereditary rings},  Journal of Algebra {\bf 307} (2007), 361-376.

%\bibitem{ABH}{\sc L. Angeleri H\" ugel, S. Bazzoni, and D. Herbera},{\sl A solution to the Baer splitting problem},to appear in Trans.\ Amer.\ Math.\ Soc.

\bibitem{AC1}
{\sc L.\ Angeleri H\" ugel and F.\ Coelho},
{\sl Infinitely generated tilting modules of finite
projective dimension}, Forum Math {\bf 13} (2001), 239-250.

%\bibitem{AC}{\sc L.\ Angeleri H\" ugel and F.\ Coelho},{\sl Infinitely generated complements to partial tilting modules},Math. Proc. Camb. Phil. Soc. {\bf 132} (2002), 89-96.

\bibitem{relml}
{\sc L. Angeleri H\" ugel,  and D. Herbera}, {\sl  
Mittag-Leffler conditions on modules}, Indiana Univ. Math. J., in press.

\bibitem{AHT}
{\sc Angeleri H\" ugel, L; Herbera, D.; Trlifaj, J.}, {\sl Tilting
modules and Gorenstein rings},
 Forum Math. {\bf 18} (2006), 217-235.

\bibitem{baerml}  {\sc L. Angeleri H\" ugel, D. Herbera, and J. Trlifaj},
{\sl Baer and Mittag-Leffler modules over tame hereditary algebras}, preprint.

\bibitem{tel}
{\sc L. Angeleri H\" ugel, J. \v Saroch, J. Trlifaj},
{\sl On the telescope conjecture for module categories},
Journal of Pure and Appl. Algebra {\bf 212} (2008),297-310.

%\bibitem{ATT}
%Algebras and Repres. Theory {\bf 4} (2001), 155-170.

\bibitem{at1}
{\sc L. Angeleri H\" ugel; J.Trlifaj},
{\sl Tilting theory and the finitistic dimension conjectures}, 
Trans. Amer. Math. Soc. {\bf 354} (2002), 4345-4358.


\bibitem{at2}
{\sc L. Angeleri H\" ugel, J. Trlifaj},  {\sl Direct limits of modules of
finite projective dimension}. In: Rings, Modules, Algebras, and
Abelian Groups. LNPAM 236   M.\ Dekker (2004), 27--44.

\bibitem{AV} {\sc  L.~Angeleri H\"ugel, H.~Valenta}:
{\sl A duality result for almost split sequences},    Colloquium
Mathematicum {\bf  80} (1999), 267-292.

\bibitem{AR}
{\sc M. Auslander; I. Reiten}, 
{\sl Applications of contravariantly finite subcategories}, 
Adv. Math. {\bf 86} (1991), 111-152.

%\bibitem{Az3}{\sc G.~Azumaya}, 
%Locally pure-projective modules, Contemp. Math. {\bf 124} (1992), 17-22.

%\bibitem{B}
%{\sc R.\ Baer},
%{\sl The subgroup of the elements of finite order of an abelian group},
%Ann. of Math. \textbf{37} (1936), 766--781.

\bibitem{B}{\sc S. Bazzoni}, Cotilting modules are pure-injective, Proc. Amer. Math. Soc. 131 (2003), 3665-3672.


\bibitem{bazher}{\sc  S.~Bazzoni and D.~Herbera},  {\sl One dimensional tilting modules are of finite
type, } Algebras and Representation Theory {\bf 11} (2008), 43-61.


\bibitem{bast}{\sc  S.~Bazzoni and J.\ \v S\v tov\' \i\v cek},  
{\sl All tilting modules are of finite
type, } Proc.\ Amer.\ Math.\ Soc. {\bf 135} (2007), 3771-3781.

%\bibitem{BK}
%{\sc A.B. Buan and  H. Krause},
%{\sl Cotilting modules over tame hereditary algebras}, 
%Pacific J. Math. {\bf 211} (2003), 41-60.

\bibitem{C1}
{\sc W.\ Crawley-Boevey}, {\sl Regular modules for tame hereditary algebras}, Proc.\ London Math.\
Soc. {\bf 62} (1991), 490--508.

\bibitem{CB}{\sc W.~Crawley-Boevey},
{\sl Locally finitely presented additive categories}, 
Comm. Algebra {\bf 22} (1994), 1644-1674.

\bibitem{C2}
{\sc W.\ Crawley-Boevey}, 
{\sl Infinite dimensional modules in the representation theory of finite dimensional algebras},
In: Algebras and modules I (ed. by I. Reiten, S. Smal\o{}  and O. Solberg), 
CMS Conf.\ Proc.\ {\bf 23} (1998), 29--54.

%\bibitem{EFS}
%{\sc P.C.\ Eklof, L.\ Fuchs, and S.\ Shelah},
%{\sl Baer modules over domains},
%Trans. Amer. Math. Soc. \textbf{322} (1990), 547-560.

\bibitem{FS}  {\sc L. Fuchs and L.Salce}, {\sl Modules over Non-Noetherian
Domains}, AMS, Providence 2001.

%%J. Algebra  144  (1991),  no. 2, 273--343. 

\bibitem{GT}  {\sc R. Goebel and J.Trlifaj}, {\sl Approximations and Endomorphism Algebras of Modules},
W.\ de Gruyter, Berlin 2006.

%\bibitem{G}
%{\sc P.\ Griffith},
%{\sl A solution to the splitting mixed group problem of Baer},
%Trans. Amer. Math. Soc. \textbf{139} (1969), 261--269.

%\bibitem{K}  {\sc I. Kaplansky}, {\sl The splitting of modules over integral domains},
%Arch.\ Math. \textbf{13} (1962), 341--343.

\bibitem{KT}
{\sc O.\ Kerner and J.\ Trlifaj},
{\sl Tilting classes over wild hereditary algebras},
J.\ Algebra. {\bf 290} (2005), 538-556.

\bibitem{KSa}{\sc H.~Krause, M.~ Saor\'{\i}n}, 
{\sl On minimal approximations of modules},   In: Trends in the representation theory of finite dimensional algebras (ed. by E.~L.~Green and B.~Huisgen-Zimmermann), Contemp. Math. {\bf 229} (1998) 227-236.


\bibitem{KS}
{\sc H.\ Krause; O.\ Solberg},
{\sl Applications of cotorsion pairs},
J.\ London Math.\ Soc. \textbf{68} (2003), 631--650.

%\bibitem{Le2}{\sc H.~Lenzing},
%Homological transfer from finitely presented to infinite modules, Lecture Notes in Math. {\bf  1006} (1983), 734-761.

%%Representations of algebras (Ottawa, ON, 1992),  339--352, CMS Conf. Proc., 14, Amer. Math. Soc., Providence, RI, 1993.

\bibitem{L1}  {\sc F. Lukas}, {\sl Infinite-dimensional modules over wild hereditary algebras},
J. London Math.\ Soc. \textbf{44} (1991), 401--419.

\bibitem{L2}  {\sc F. Lukas}, {\sl A class of infinite-rank modules over tame hereditary algebras},
J. Algebra \textbf{158} (1993), 18--30.

%\bibitem{O1}{\sc F. Okoh},
%{\sl Cotorsion modules over tame finite-dimensional hereditary algebras}, in Springer Lecture
%Notes in Math.~{\bf 903} (1980), 263--269.

%\bibitem{O2}  {\sc F. Okoh}, {\sl Baer modules},
%J. Algebra \textbf{77} (1982), 402--410.

%\bibitem{Okohsep}  {\sc F. Okoh}, {\sl Separable modules over finite-dimensional algebras},
%J.   Algebra \textbf{116} (1988), 400--414.

%\bibitem{O3}  {\sc F. Okoh}, {\sl Bouquets of Baer modules},
%J. Pure Appl. Algebra \textbf{93} (1994), 297--310.

%\bibitem{RG}{\sc M.\ Raynaud et L.\ Gruson}, {\sl Crit\`eres de platitude et de
%projectivit\'e}, Invent. Math. {\bf 13}(1971), 1--89.

\bibitem{RR} {\sc I.\ Reiten and C.M.\ Ringel}, {\sl Infinite dimensional representations of canonical algebras},
Canad.\ J.\ Math. \textbf{58} (2006), 180--224.

\bibitem{R} {\sc C.M.\ Ringel}, {\sl Infinite dimensional representations of finite dimensional hereditary algebras},
Symposia Math. \textbf{23} (1979), 321--412.

\bibitem{RSpe} {\sc C.M.\ Ringel}, {\sl The Ziegler spectrum of a tame hereditary algebra},
Coll Math. \textbf{76} (1998), 105--115.

%%Cambridge Univ.\ Press, Cambridge 1985.

%\bibitem{S2}{\sc A. Schofield},
%{\sl Universal localization for hereditary rings and quivers}, in Springer Lecture Notes in
%Math.~{\bf 1197} (1986), 149--164.

\bibitem{SarSt} 
{\sc J. \v Saroch and J.\ \v S\v tov\' \i\v cek}, 
{\sl The countable telescope conjecture for module categories}, preprint.

\bibitem{S}{\sc J.\ \v S\v tov\' \i\v cek}, {\sl All n-cotilting modules are pure-injective}, Proc. Amer. Math. Soc. 134 (2006), 1891-1897.

\bibitem{ZH1}{\sc B.~Zimmermann-Huisgen}, {\sl Strong preinjective partitions and representation type of artinian rings}, Proc. Amer. Math. Soc. {\bf 109} (1990), 309-322.

%\bibitem{ZHZ}{\sc B.~Zimmermann-Huisgen, W.~Zimmermann}, On the sparsity of representations of rings of pure global dimension zero, Trans. Amer. %Math. Soc. \underline{320} (1990), 695-711.

%\bibitem{Z}  {\sc W. Zimmermann}, {\sl On locally pure-injective modules},
%J.\ Pure Appl.\ Algebra \textbf{166} (2002), 337--357.

\end{thebibliography}
\end{document}